\documentclass[10pt,number,preprint,times]{elsarticle}
\usepackage{physics}
\makeatletter

\usepackage{amssymb}
\usepackage{atbegshi}
\AtBeginDocument{\AtBeginShipoutNext{\AtBeginShipoutDiscard}}
\usepackage[utf8]{inputenc} 
\usepackage{alltt}
\usepackage{lipsum}
\usepackage[linesnumbered,ruled,vlined]{algorithm2e}
\usepackage{amsfonts}
\usepackage{mathtools}
\usepackage{amsmath}
\usepackage{times}
\usepackage{tikz-cd}
\usepackage{manyfoot}%
\makeatletter
\makeatletter
\def\ps@pprintTitle{%
	\let\@oddhead\@empty
	\let\@evenhead\@empty
	\let\@oddfoot\@empty
	\let\@evenfoot\@oddfoot
}
\makeatother
\usepackage{booktabs}
\usepackage[T1]{fontenc}
\usepackage[english,french]{babel}
\usepackage{lmodern}
\usepackage[a4paper]{geometry}
\usepackage{epstopdf}
\usepackage{mathrsfs}
\usepackage{latexsym,amssymb,amsfonts}
\usepackage{enumerate}
\usepackage{relsize,exscale}
\usepackage{dsfont}
\usepackage{microtype}
\usepackage{xcolor}
\usepackage{enumitem} 
\usepackage{graphicx}
\usepackage{float}
\usepackage{upgreek}
\usepackage{caption}
\usepackage{subfigure}

\usepackage[labelfont=bf]{caption}
\usepackage{xcolor}
\definecolor{mred0}{rgb}{0,0.406,0.596}
\definecolor{mcyan}{rgb}{0,0.501,0.501}
\definecolor{mcyan}{rgb}{0,0.501,0.501}
\definecolor{mblue}{rgb}{0,0.406,0.796}
\definecolor{mcgreen}{rgb}{0.180,0.545,0.41}
\definecolor{mred0}{rgb}{0,0.406,0.596}
\definecolor{mred3}{rgb}{0.6,0,0.8}
\definecolor{DarkGreen}{rgb}{0.278,0.701,0.913} 
\definecolor{mred1}{rgb}{0.180,0.545,0.341}

\usepackage{longtable}

\numberwithin{equation}{section} 

\usepackage{amsthm}
\newtheorem{theorem}{Theorem}[section]

\newtheorem{remark}{Remark}[section]

\usepackage{float}
\usepackage{multicol}
\usepackage[colorlinks=true,linkcolor=blue,citecolor=green,urlcolor=red]{hyperref}
\usepackage[figurename=Fig.]{caption}  

\journal{Communications in Nonlinear Science and Numerical Simulation}

\date{}
\begin{document}
	\maketitle
	\selectlanguage{english}
		\begin{abstract}
A crucial feature of reaction–diffusion epidemic models is the incidence function, which characterizes disease transmission dynamics. Over the past few decades, many studies have investigated the behavior of such models under various incidence functions. However, the question of how to appropriately select a suitable incidence function remains largely unexplored. This paper addresses this issue by proposing an intuitive theoretical framework that recasts the original problem as determining the contributions of different incidence functions to the dynamics based on given observations. Specifically, the choice of an incidence function is linked to a weight assigned to it. Mathematically, this leads to a PDE-constrained optimization problem, where the objective is to identify the weights of a convex combination of multiple incidence functions that best approximate the experimental observations generated by the model. We first establish the Fréchet differentiability of the parameter-to-state operator and then derive the optimality conditions for the weights using a suitable adjoint problem. Finally, we illustrate our approach with a numerical example based on the Landweber iteration algorithm.\\[1ex]
		\textit{Keywords}: 	Partial differential equations, PDE-constrained optimization, Epidemic model, Parameter identification, Numerical simulation\\
		\textit{2020 MSC}: 35K57, 49K20, 49J20, 92D30, 62P10
	\end{abstract}
	\begin{frontmatter}
\title{A note on quantifying the contributions of incidence functions in spatio-temporal epidemic models}

\author{Mohamed Mehdaoui \corref{mycorrespondingauthor}}
\cortext[mycorrespondingauthor]{Corresponding author:  m.mehdaoui@ueuromed.org}
\address{Euromed University of Fez, UEMF, 30 000, Fez, Morocco}

\author{Mouhcine Tilioua}
\address{MAMCS Group, FST Errachidia, Moulay Ismail University of Meknes, P.O. Box 509, Boutalamine 52000, Errachidia, Morocco}

\ead{m.tilioua@umi.ac.ma}
\end{frontmatter}

	\section{Introduction and motivation}\label{intro}
\subsection{Background}
Throughout centuries, reaction--diffusion models have shown to be among the simplest yet effective ways of capturing the spatial random mobility of individuals in epidemic scenarios \cite{capasso1978global,capasso1993mathematical,coronel2019note,chen2021global,mehdaoui2023analysis,parkinson2023analysis,mehdaoui2023optimal,colli2024global,mehdaoui2024optimal,chang2024time,mehdaoui2024well}. Following the pioneering work by Capasso \cite{capasso1978global}, various extensions of the so-called Susceptible--Infected--Recovered (SIR) model have been proposed. Given an open bounded set $\mathcal{O} \subset \mathbb{R}^N$ ($N \in \{1,2,3\}$) with smooth boundary $\partial \mathcal{O}$ and $T>0$, in this paper, we devote our focus to the following class:
\begin{equation}\label{proto}
	\begin{cases}
		\begin{split}
			&S_t-d_1 \Delta S=-\sum_{i=1}^{m} \theta_i f_i(S,I), \quad &&\text{ in } Q_T:=\mathcal{O} \times (0,T),\\
			&I_t-d_2 \Delta I=\sum_{i=1}^{m} \theta_i f_i(S,I)-\gamma I, \quad &&\text{ in } Q_T,\\
			&R_t-d_3 \Delta R=\gamma I, \quad &&\text{ in } Q_T,\\
			&\nabla S \cdot \overrightarrow{n}=\nabla I \cdot \overrightarrow{n}=\nabla R \cdot \overrightarrow{n}=0, \quad &&\text{ on } \Sigma_T:=\partial \mathcal{O} \times (0,T),\\
			&(S(.,0),I(.,0),R(.,0))=(S_0,I_0,R_0), \quad &&\text{ in } \mathcal{O}.
		\end{split}
	\end{cases}
\end{equation}
Herein, $S$, $I$ and $R$ denote the densities of the susceptible, infected and recovered populations diffusing at positive rates $d_1$, $d_2$ and $d_3$ and starting from positive initial states $S_0$, $I_0$ and $R_0$, respectively. Moreover, $\overrightarrow{n}$ stands for the outward normal unit vector, while the notations $u_t$, $\Delta u$ and $\nabla u$ refer to the time partial derivative, the laplacian and the gradient of a given function $u$. Additionally, $\gamma$ stands for the positive  recovery rate.  On the other hand, given $m \in \mathbb{N}^*$ and $\left(\theta_1,\cdots, \theta_m\right) \in [0,\infty)^m$ such that $\displaystyle \sum_{i=1}^{m} \theta_i=1$, the function $(S,I) \mapsto \displaystyle \sum_{i=1}^{m} \theta_i f_i(S,I)$ describes the disease incidence between the susceptible and the infected classes, where for all $i \in \{1,\cdots,m\}$, $f_i : [0,\infty)^2 \rightarrow [0,\infty)$ denotes the $i^{th}$  incidence function contributing to the dynamics at a rate $\theta_i$ and is assumed to be continuously differentiable and satisfying the following standard condition:
\begin{equation*}
	f_i(0,I)=f_i(S,0)=0, \forall (S,I)\in [0,\infty)^2.
\end{equation*}
\begin{remark} The parameters \(\theta_1, \cdots, \theta_m\) are mathematical weights, but they also have biological significance. Specifically, \(\theta_i\) captures the degree at which the biological factor modeled by \(f_i\) influences the disease dynamics. The choice of \(f_i\) should be based on empirical data or biological theory, which may vary depending on the context. \end{remark}
\subsection{The Case of One Incidence Function $(m=1)$: A Key Limitation}
Let us consider the particular case where \( f_i \equiv f \). In this case, Model~\eqref{proto} reduces to the classical scenario in which only a single type of incidence function contributes to the dynamics. From a practical point of view, such a case is overly restrictive and provides only one degree of freedom to describe the complex dynamics of infectious diseases. This limitation becomes particularly pronounced in real-world scenarios, as observed during the COVID-19 pandemic and other recent outbreaks, where multiple interconnected factors simultaneously influenced disease transmission. For instance, during the COVID-19 pandemic, disease spread was not solely determined by direct interactions between susceptible and infected individuals. Various interventions, such as mask mandates, lockdowns, and vaccination campaigns, substantially modified transmission dynamics. Relying on a standard bilinear function,
$(S, I) \mapsto \beta SI$
\cite{mehdaoui2023dynamical,mmehdaoui2023analysis,della2025optimality} may fail to account for these dynamic behavioral and policy-driven changes. Similarly, when considering the protective behaviors adopted by individuals (e.g., social distancing and hand hygiene), a Holling-type function,
$(S, I) \mapsto \dfrac{\beta SI}{1 + aI}\; (a>0)$
\cite{buonomo2008dynamics,li2014traveling}, might be more appropriate, but it still neglects the effects of other critical factors, such as population density or repeated exposure risks. In densely populated regions or during superspreader events, the interplay between high population contact rates and limited healthcare resources can further complicate transmission dynamics. Such scenarios may require alternative choices such as the Beddington-DeAngelis function,
$(S, I) \mapsto \dfrac{\beta SI}{1 + aI + bS} \; (a,b>0)$
\cite{huang2009global}, to capture mutual interference effects and saturation phenomena. Alternatively, the influence of the total population size, particularly relevant in COVID-19 outbreaks within urban centers, can be modeled for example by using
$(S, I) \mapsto \dfrac{\beta SI}{S + I + R}$
\cite{zhang2021stationary}. Moreover, recent diseases, including COVID-19, have highlighted the role of repeated exposure risks in transmission. The disproportionate impact of certain environments, such as workplaces, schools, or households, can be modeled with incidence functions of the form
$(S, I) \mapsto \beta SI(1 + kI) \; (k>0)$
\cite{mehdaoui2024optimal,buonomo2008dynamics}, which account for enhanced risks due to double or cumulative exposures. These examples underscore the inadequacy of relying on a single incidence function in the modeling of modern infectious diseases. The variety of factors influencing transmission, ranging from individual behavior to policy measures, population density, and healthcare infrastructure, calls for the adoption of flexible and multifactorial models. Simplistic assumptions risk failing to capture critical dynamics, especially in complex outbreaks such as COVID-19, where multiple mechanisms are at play simultaneously.


\subsection{Towards the Open Question of Selecting the Suitable Incidence Function}
In comparison to questions related to well-posedness and asymptotic behavior, considerable attention to identifying the parameters and capturing disease transmission has not been fully addressed, as evidenced by the limited literature from Xiang and Liu \cite{xiang2015solving}, Coronel et al \cite{coronel2019note,coronel2021existence}, and more recently Chang et al. \cite{chang2024time}. While previous literature has focused on identifying the disease transmission rate, it assumes a fixed incidence function, $f_i \equiv f$, without simultaneously ensuring that this choice adequately reflects the observed dynamics. This approach implicitly prioritizes the biological factors captured by the chosen incidence function, at the expense of potentially overlooking other significant factors that could be represented by alternative forms of $f_i$. Furthermore, in environments where multiple factors interact, such as public health interventions, individual behaviors, and population structure, the risk of using a single, fixed incidence function becomes more pronounced. It is crucial to recognize that different incidence functions, when calibrated, can yield substantially different results, leading to disparate model predictions. Without carefully considering the contributions of each factor to the disease dynamics, the model may inaccurately reflect the true transmission process. In particular, for complex systems where various mechanisms coexist, such as those observed in the COVID-19 pandemic, failing to account for the weight of each incidence function could lead to a substantial mismatch between the calibrated model and real-world observations. This mismatch is often caused by neglecting critical factors like behavioral changes, intervention effects, or population-specific variables. Thus, it is essential to identify the weight that each incidence function carries in the dynamics when proceeding to calibration. By doing so, we ensure that the chosen incidence functions are appropriate and representative of the observed data, preventing the risk of erroneous conclusions or overly simplistic models that fail to capture the full complexity of disease transmission. In this paper, we propose an intuitive approach to address this issue, which we believe should be carried out simultaneously with the disease transmission identification step developed in \cite{coronel2019note,chang2024time,xiang2015solving,coronel2021existence}. More precisely, we address the following theoretical question:
\begin{enumerate}[label=\textcolor{blue}{(Q)}]  
	\item \label{qs} \textit{At which weight does the \(i^{th}\) incidence function contribute to the dynamics of Model \eqref{proto}?}  
\end{enumerate}

\subsection{Addressing the Limitation with a New Approach}
From the mathematical standpoint, Question \ref{qs} can be recast into the following PDE-constrained optimization problem: 
\begin{equation}\label{prob}
	\begin{cases}
		\begin{split}
			&\underset{\theta \in \mathscr{U}^\theta_{ad}}{\min}\displaystyle \mathcal{J}\left(\theta,S,I\right):= \int_{Q_T} \vert S-S^{ob} \vert^2dxdt+\int_{Q_T} \vert I-I^{ob} \vert^2dxdt+\dfrac{\sigma}{2} \Vert \theta \Vert_{\mathbb{R}^m}^2,\\
			&\text{such that } (S,I)  \text{ satisfies Model } \eqref{proto} \text{ in a strong sense},
		\end{split} 
	\end{cases} 
\end{equation}
where $S^{ob}, I^{ob} \in L^2(Q_T)$ are given observations and $\sigma>0$. On the other hand,
$$
\mathscr{U}^\theta_{ad}:=\left\{\theta \in \mathbb{R}^m : \quad \theta_i \geq 0, \forall i \in \{1,\cdots,m\} \text{ and } \sum_{i=1}^m \theta_i=1 \right\}.
$$
To justify the constraint in Problem \eqref{prob}, we consider the following parameter-to-state operator
\begin{align*}
	\mathscr{F}: \;&\mathscr{U}^\theta_{ad} \longrightarrow L^2(Q_T)^3\\
	& \theta \mapsto (S_\theta,I_\theta,R_\theta),
\end{align*}
such that $(S_\theta,I_\theta,R_\theta)$ is the strong solution to Model \eqref{proto} with the parameter $\theta$. 
Assuming hereafter that 

\begin{equation}
	(S_0,I_0, R_0) \in \left\{u \in H^2(\mathcal{O})^3: \quad \nabla u_i \cdot \overrightarrow{n}=0  \text{ on } \partial \mathcal{O}, \forall i \in \{1,2,3\}\right\},
\end{equation}
and proceeding along the same lines of \cite[Theorems 1 and 3]{mehdaoui2023optimal}, it can be proved that $\mathscr{F}$ is well-defined. Moreover, the following regularity holds:
$$
\mathscr{F}(\theta) \in L^2(0,T;H^2(\mathcal{O}))^3  \cap L^\infty(0,T;H^1({\mathcal{O}}))^3  \cap L^\infty(Q_T)^3.
$$
Moreover, there exists a positive constant $C_1$ independent of $\theta$, such that 
\begin{equation}\label{es}
	\Vert \mathscr{F}(\theta) \Vert_{L^2(0,T;H^2(\mathcal{O}))^3}  + \Vert \mathscr{F}(\theta) \Vert_{L^\infty(0,T;H^1({\mathcal{O}}))^3}+\Vert \mathscr{F}(\theta) \Vert_{L^\infty(Q_T)^3}  \leq C_1.
\end{equation}
\noindent
Consequently, Problem \eqref{prob} reduces to
\begin{equation}\label{prob2}
	\underset{\theta \in \mathscr{U}^\theta_{ad}}{\min}\displaystyle \mathcal{J}\left(\theta,\mathscr{F}(\theta)\right):= \int_{Q_T} \vert S_\theta-S^{ob} \vert^2dxdt+\int_{Q_T} \vert I_\theta-I^{ob} \vert^2dxdt+\dfrac{\sigma}{2} \Vert \theta \Vert_{\mathbb{R}^m}^2.
\end{equation}

We arrange the remaining of this paper as follows. In Section \ref{section2}, we establish the Fréchet differentiability of the parameter-to-state operator $\mathscr{F}$. In Section \ref{section3}, we derive the necessary optimality conditions for optimal local solutions to Problem \eqref{prob2}. In Section \ref{section4}, we provide the outcomes of a numerical simulation based on the Landweber iteration algorithm, in order to support our proposed theoretical approach. Finally, we devote Section \ref{section5} to exploring some future directions.
\section{Fréchet differentiabiliy of the parameter-to-state operator}
\label{section2}
We state the main result of this section as follows.
\begin{theorem}\label{frechet}
	It holds that $\mathscr{F}$ is Fréchet-differentiable with a Fréchet derivative given by $\mathscr{F}^\prime(\theta)(\tilde{\theta})=(\overline{S},\overline{I},\overline{R}),\; \forall \theta \in \mathscr{U}_{ad}^\theta, \forall \tilde{\theta} \in \mathbb{R}^m$, where
	\begin{equation}\label{protosens}
		\begin{cases}
			\begin{split}
				&\overline{S}_t-d_1 \Delta \overline{S}=-\sum_{i=1}^{m} \tilde{\theta}_i f_i(S_\theta,I_\theta)-\sum_{i=1}^{m} {\theta}_i \left(\partial_S f_i(S_\theta,I_\theta)\overline{S}+\partial_I f_i(S_\theta,I_\theta)\overline{I} \right), \quad &&\text{ in } Q_T,\\
				&\overline{I}_t-d_2 \Delta \overline{I}=\sum_{i=1}^{m} \tilde{\theta}_i f_i(S_\theta,I_\theta)+\sum_{i=1}^{m} {\theta}_i \left(\partial_S f_i(S_\theta,I_\theta)\overline{S}+\partial_I f_i(S_\theta,I_\theta)\overline{I} \right)-\gamma \overline{I}, \quad &&\text{ in } Q_T,\\
				&\overline{R}_t-d_3 \Delta \overline{R}=\gamma \overline{I}, \quad &&\text{ in } Q_T,\\
				&\nabla \overline{S} \cdot \overrightarrow{n}=\nabla \overline{I} \cdot \overrightarrow{n}=\nabla \overline{R} \cdot \overrightarrow{n}=0, \quad &&\text{ on } \Sigma_T,\\
				&(\overline{S}(.,0),\overline{I}(.,0),\overline{R}(.,0))=(0,0,0), \quad &&\text{ in } \mathcal{O}.
			\end{split}
		\end{cases}
	\end{equation}
\end{theorem}
\begin{proof}
	Let $\tilde{\theta} \in \mathbb{R}^m$ such that $\theta+\tilde{\theta}\in \mathscr{U}^\theta_{ad}$. 
	We denote 
	$\mathscr{F}(\theta+\tilde{\theta} )\overset{\Delta}{=}(S^{\theta+\tilde{\theta}},I^{\theta+\tilde{\theta}},R^{\theta+\tilde{\theta}})$
	and
	$\mathscr{F}(\theta)\overset{\Delta}{=}(S^{\theta},I^{\theta},R^{\theta}).$
	We further introduce the intermediate variable:
	$$
	(Z_1,Z_2,Z_3):=(S^{\theta+\tilde{\theta}},I^{\theta+\tilde{\theta}},R^{\theta+\tilde{\theta}}) -(S^{\theta},I^{\theta},R^{\theta})-(\overline{S},\overline{I},\overline{R}).
	$$
	By computing, we obtain
	\begin{small}
		\begin{equation}\label{b}
			\begin{cases}
				\begin{split}
					{Z_1}_t&-d_1 \Delta {Z_1}=-\sum_{i=1}^{m} \theta_i \left(f_i(S^{\theta+\tilde{\theta}},I^{\theta+\tilde{\theta}})-f_i(S^{\theta},I^{\theta})\right)-\sum_{i=1}^{m} \tilde{\theta}_i \left(f_i(S^{\theta+\tilde{\theta}},I^{\theta+\tilde{\theta}})\right.\\
					&\left.-f_i(S^{\theta},I^{\theta})\right)+\sum_{i=1}^{m} {\theta}_i \left(\partial_S f_i(S_\theta,I_\theta)\overline{S}+\partial_I f_i(S_\theta,I_\theta)\overline{I} \right),  &&\text{ in } Q_T,\\[1.6ex]
					{Z_2}_t&-d_2 \Delta {Z_2}=\sum_{i=1}^{m} \theta_i \left(f_i(S^{\theta+\tilde{\theta}},I^{\theta+\tilde{\theta}})-f_i(S^{\theta},I^{\theta})\right)+\sum_{i=1}^{m} \tilde{\theta}_i \left(f_i(S^{\theta+\tilde{\theta}},I^{\theta+\tilde{\theta}})\right.\\
					&\left.-f_i(S^{\theta},I^{\theta})\right)+\sum_{i=1}^{m} {\theta}_i \left(\partial_S f_i(S_\theta,I_\theta)\overline{S}+\partial_I f_i(S_\theta,I_\theta)\overline{I} \right)\\
					&-\gamma {Z_2}, \quad &&\text{ in } Q_T,\\[1.6ex]
					{Z_3}_t&-d_3 \Delta {Z_3}=\gamma {Z_2}, \quad &&\text{ in } Q_T,\\
					\nabla {Z_1} &\cdot \overrightarrow{n}=\nabla {Z_2} \cdot \overrightarrow{n}=\nabla {Z_3} \cdot \overrightarrow{n}=0, \quad &&\text{ on } \Sigma_T,\\
					({Z_1}&(.,0),{Z_2}(.,0),{Z_3}(.,0))=(0,0,0), \quad &&\text{ in } \mathcal{O}.
				\end{split}
			\end{cases}
		\end{equation}
	\end{small}
	Now, multiplying the first, second and third equations of Problem \eqref{b} by $Z_1$, $Z_2$ and ${Z_3}$, respectively and integrating over $\mathcal{O}$ and recalling Estimate \eqref{es}, we use Cauchy-Schwarz inequality along with the mean value theorem on $f_i (i \in \{1,\cdots,m\})$, to eventually obtain that
	\begin{equation}
		\begin{split}
			\dfrac{1}{2} \dfrac{d}{dt} \left(\Vert Z_1 \Vert_{L^2(\mathcal{O})}^2+\Vert Z_2 \Vert_{L^2(\mathcal{O})}^2+\Vert Z_3 \Vert_{L^2(\mathcal{O})}^2\right) &\leq C_2 \left(\Vert Z_1 \Vert_{L^2(\mathcal{O})}^2+\Vert Z_2 \Vert_{L^2(\mathcal{O})}^2+\Vert Z_3 \Vert_{L^2(\mathcal{O})}^2\right)\\
			&+C_3 \Vert \tilde{\theta} \Vert_{\mathbb{R}^m},
		\end{split}
	\end{equation}
	where $C_2$ and $C_3$ are positive constants independent of $\theta$.
	
	\noindent
	By directly applying Gronwall's inequality and recalling that $$({Z_1}(.,0),{Z_2}(.,0),{Z_3}(.,0))=(0,0,0),$$ 
	we acquire that 
	$$
	\Vert Z_1 \Vert_{L^\infty(0,T;L^2(\mathcal{O}))}^2+\Vert Z_2 \Vert_{L^\infty(0,T;L^2(\mathcal{O}))}^2+\Vert Z_3 \Vert_{L^\infty(0,T;L^2(\mathcal{O}))}^2 \leq 2 C_3 \exp(C_2T) \Vert \tilde{\theta} \Vert_{\mathbb{R}^m}.
	$$
	Thus, the result is obtained by letting $\Vert \tilde{\theta} \Vert_{\mathbb{R}^m} \rightarrow 0$ and recalling the definition  of the Fréchet derivative.
\end{proof}
\section{Existence of an optimal solution and first-order necessary optimality conditions}
\label{section3}
We are now ready to derive the necessary optimality conditions satisfied by optimal local solutions to Problem \eqref{prob2}. Prior to that, let us first state a standard intermediate result ensuring that such solutions exist.
\begin{theorem}
	Problem \eqref{prob2} has at least a global solution $\theta^* \in \mathscr{U}_{ad}^\theta$.
\end{theorem}
\begin{proof}
	The proof is standard and builds on the technique of minimizing sequences (see e.g. \cite[Theorem 3.1]{mehdaoui2024optimal}). We thus omit it here for brevity.	
\end{proof}

\noindent
The derivation of the necessary optimality conditions satisfied by local solutions to Problem \eqref{prob2} is obtained through the following well-posed adjoint problem:
\begin{equation}\label{adj}
	\begin{cases}
		\begin{split}
			&-{P_1}_t-d_1 \Delta {P_1}=\sum_{i=1}^{m} {\theta}_i \partial_S f_i(S_\theta^*,I_\theta^*)(P_2-P_1)+2(S_\theta^*-S^{ob}), \quad &&\text{ in } Q_T,\\[1.5ex]
			&-{P_2}_t-d_2 \Delta {P_2}=\sum_{i=1}^{m} {\theta}_i \partial_I f_i(S_\theta^*,I_\theta^*)(P_2-P_1)+\gamma (P_3-P_2)+2(I_\theta^*-I^{ob}), \quad &&\text{ in } Q_T,\\
			&-{P_3}_t-d_3 \Delta {P_3}=0, \quad &&\text{ in } Q_T,\\[1.5ex]
			&\nabla {P_1} \cdot \overrightarrow{n}=\nabla {P_2} \cdot \overrightarrow{n}=\nabla {P_3} \cdot \overrightarrow{n}=0, \quad &&\text{ on } \Sigma_T,\\[1.5ex]
			&({P_1}(.,T),{P_2}(.,T),{P_3}(.,T))=(0,0,0), \quad &&\text{ in } \mathcal{O}.
		\end{split}
	\end{cases}
\end{equation}
\begin{remark}
	Note that Problem \eqref{prob2} is non-convex, due to the non-linearity of the parameter-to-state operator $\mathscr{F}$. Generally, when it comes to such problems, one only establishes the necessary optimality conditions for local solutions \cite{mehdaoui2024optimal,mehdaoui2025state,mehdaoui2024new}. Additionally, when it comes to numerical implementations, the chosen minimization algorithm is usually expected to only generate local solutions. For a more detailed discussion on this subject, we refer the reader for example to \cite[p. 221]{troltzsch2010optimal}. 
\end{remark}
\noindent
The local characterization of the optimal (local) solution $\theta^*$ is stated as follows. 
\begin{theorem}
	Let $\theta^*$ be an optimal solution to Problem \eqref{prob2}. Then, the following variational inequality holds:
	\begin{equation}
		\sum_{i=1}^m \tilde{\theta}_i \left(\int_{Q_T} f_i(S_\theta^*,I_\theta^*)(P_2-P_1)dxdt+\sigma \theta^*_i\right) \geq 0, \quad \forall \tilde{\theta} \in \mathbb{R}^m,
	\end{equation}
	where $(S_\theta^*,I_\theta^*)$ denotes the component corresponding to the strong solution to Model \eqref{proto} with the parameter $\theta^*$, while $(P_1,P_2)$ denotes the component of the strong solution corresponding to Problem \eqref{adj}.
\end{theorem}

\begin{proof}
	Let $\epsilon>0$ and $\theta^*$ be an optimal solution to Problem \eqref{prob2}. Moreover, let $\tilde{\theta} \in \mathbb{R}^m$ such that $$\theta^\epsilon:=\theta^*+\epsilon \tilde{\theta} \in \mathscr{U}_{ad}^\theta.$$
	Denote by $(S^\epsilon_\theta,I^\epsilon_\theta,R^\epsilon_\theta)$ the corresponding solution to Model \eqref{proto} with a corresponding parameter $\theta^\epsilon$, and by $(S_\theta^*,I_\theta^*,R_\theta^*)$ the one corresponding to $\theta^*$. Moreover, set for simplicity
	$$
	\overline{S}^\epsilon\overset{\Delta}{=}\dfrac{S^\epsilon_\theta-S^*_\theta}{\epsilon}, \; \overline{I}^\epsilon\overset{\Delta}{=}\dfrac{I^\epsilon_\theta-I^*_\theta}{\epsilon}, \;
	\overline{R}^\epsilon\overset{\Delta}{=}\dfrac{R^\epsilon_\theta-R^*_\theta}{\epsilon}.
	$$
	By computing, we obtain that 
	\begin{equation}\label{protosenses}
		\begin{cases}
			\begin{split}
				\overline{S}^\epsilon_t&-d_1 \Delta \overline{S}^\epsilon=-\sum_{i=1}^{m} \tilde{\theta}_i f_i(S_\theta^\epsilon,I_\theta^\epsilon)-\sum_{i=1}^{m} {\theta}^*_i \dfrac{f_i(S^\epsilon_\theta,I^\epsilon_\theta)-f_i(S_\theta^*,I_\theta^\epsilon)}{S_\theta^\epsilon-S_\theta^*}\overline{S}^\epsilon\\
				&-\sum_{i=1}^{m} {\theta}^*_i \dfrac{f_i(S_\theta^*,I_\theta^\epsilon)-f_i(S_\theta^*,I_\theta^*)}{I^\epsilon_\theta-I_\theta^*}\overline{I}^\epsilon, \quad &&\text{ in } Q_T,\\
				\overline{I}^\epsilon_t&-d_2 \Delta \overline{I}^\epsilon=\sum_{i=1}^{m} \tilde{\theta}_i f_i(S_\theta^\epsilon,I_\theta^\epsilon)+\sum_{i=1}^{m} {\theta}^*_i \dfrac{f_i(S^\epsilon_\theta,I^\epsilon_\theta)-f_i(S_\theta^*,I_\theta^\epsilon)}{S_\theta^\epsilon-S_\theta^*}\overline{S}^\epsilon\\[1ex]
				&+\sum_{i=1}^{m} {\theta}^*_i \dfrac{f_i(S_\theta^*,I_\theta^\epsilon)-f_i(S_\theta^*,I_\theta^*)}{I^\epsilon_\theta-I_\theta^*}\overline{I}^\epsilon-\gamma \overline{I}^\epsilon , \quad &&\text{ in } Q_T,\\[1ex]
				\overline{R}^\epsilon_t&-d_3 \Delta \overline{R}^\epsilon=\gamma \overline{I}^\epsilon, \quad &&\text{ in } Q_T,\\[1ex]
				\nabla &\overline{S}^\epsilon \cdot \overrightarrow{n}=\nabla \overline{I}^\epsilon \cdot \overrightarrow{n}=\nabla \overline{R}^\epsilon \cdot \overrightarrow{n}=0, \quad &&\text{ on } \Sigma_T,\\
				(\overline{S}^\epsilon&(.,0),\overline{I}^\epsilon(.,0),\overline{R}^\epsilon(.,0))=(0,0,0), \quad &&\text{ in } \mathcal{O}.
			\end{split}
		\end{cases}
	\end{equation}
	By local optimality of $\theta^*$, we use Theorem \ref{frechet} and recall the definition of $\mathcal{J}$ to obtain that
	$$
	\underset{\epsilon \rightarrow 0}{\lim} \bigg \Vert \frac{S^\epsilon_\theta-S_\theta^*}{\epsilon}-\overline{S} \bigg \Vert_{L^2(Q_T)}=\underset{\epsilon \rightarrow 0}{\lim} \bigg \Vert \frac{I^\epsilon_\theta-I_\theta^*}{\epsilon}-\overline{I} \bigg \Vert_{L^2(Q_T)}=0,
	$$
	and so
	\begin{equation}\label{ineq}
		\begin{split}
			\underset{\epsilon \rightarrow 0}{\lim}\dfrac{\mathscr{J}(\theta^\epsilon)-\mathscr{J}(\theta^*)}{\epsilon}&=2\bigg\langle S^*_\theta-S^{ob},\overline{S} \bigg \rangle_{L^2(Q_T)\times L^2(Q_T)}+2\bigg\langle I^*_\theta-I^{ob},\overline{I} \bigg \rangle_{L^2(Q_T)\times L^2(Q_T)}\\
			&+\sigma \langle \theta^*,\tilde{\theta} \rangle_{\mathbb{R}^m\times \mathbb{R}^m} \geq 0.
		\end{split}
	\end{equation}
	Now, we multiply Equations $\eqref{protosens}_1$,  $\eqref{protosens}_2$ and $\eqref{protosens}_3$ by $P_1$, $P_2$ and $P_3$, respectively and sum up the resulting equations to eventually acquire the following identity:
	\begin{align*}
		2\bigg\langle S^*_\theta-S^{ob},\overline{S} \bigg \rangle_{L^2(Q_T)\times L^2(Q_T)}&+2\bigg\langle I^*_\theta-I^{ob},\overline{I} \bigg \rangle_{L^2(Q_T)\times L^2(Q_T)}\\
		&=\sum_{i=1}^m \tilde{\theta}_i \int_{Q_T} f_i(S_\theta^*,I_\theta^*)(P_2-P_1)dxdt.
	\end{align*}
	The result is concluded by injecting the previous identity into Inequality \eqref{ineq}.
\end{proof}

\section{A numerical example}
\label{section4}
We provide a numerical implementation of our proposed approach.  First, we consider a one dimensional domain $\mathcal{O}:=(0,2),$ a time horizon $(0,T):=(0,10)$ and an initial state $(S_0,I_0,R_0)=(0.85,0.15,0).$  Additionally, we set $\gamma=0.4, d_1=d_2=0.5,$ and  $\sigma=0.0001.$ As for the number of incidence functions, we set $m=3$ and $f_1(S,I):=0.4 SI;\; f_2(S,I)=\frac{0.4 SI}{1+I};\; f_3(S,I)=\frac{0.4 SI}{S+I+R}.$

\noindent
We numerically solve Problem \eqref{prob2} based on Algorithm \ref{myalgo}. Figures \ref{fig1} and \ref{fig2} show the obtained densities of the susceptible and infected populations, separately obtained for each incidence function (case m=1). It can be seen that although the disease transmission rate $\beta$ is the same, the generated densities are different, especially when it comes to the standard bilinear function. This confirms that, the calibration of the model with different incidence functions can lead to different computed solutions, some of which may not closely align with the observations. On the other hand, from Fig. \ref{fig3}, we observe that Algorithm \ref{myalgo} successfully identified the optimal weights allowing to minimize the cost functional J. However, it can be seen that the number of iterations required for such an outcome is higher, which is on the one hand due to the Landweber algorithm which is known to be slow, especially for a lower tolerance $\varepsilon$ and in the case where the optimized variable acts in a bilinear way on the state, which is the case here. As for the computed solution to Model 1, in the case where the weights are assigned the optimal values $\theta_1^*$, $\theta_2^*$ and $\theta_3^*$, it can be seen from Figs. \ref{fig4} and \ref{fig5} that the densities of the susceptible and infected populations closely align with the experimental observations. Finally, by taking the convex combination of the chosen incidence functions using the optimal weights $\theta_1^*$, $\theta_2^*$ and $\theta_3^*$, we obtain the incidence function that corresponds to the experimental data as shown by Fig. \ref{fig20}.


\begin{algorithm}[H]
	\label{myalgo}
	\caption{Landweber iteration algorithm solving Problem \eqref{prob2}}
	
	\textbf{Input:} Initial guess $\theta^0$ and tolerance $\varepsilon = 0.0001$.
	
	\textbf{Output:} Identified weights $\theta^*_1 = 0.245$, $\theta^*_2 = 0.294$, $\theta^*_3 = 0.461$.
	
	\begin{enumerate}
		\item Solve the direct problem with $\theta_1 = 0.2$, $\theta_2 = 0.3$, $\theta_3 = 0.5$ to obtain $(S, I)$.
		\item Generate experimental observations $(S^{ob}, I^{ob})$ by adding random noise to $(S, I)$.
		\item Set $k = 0$.
		\item Repeat until $\lvert \mathscr{J}(\theta^{k+1}) - \mathscr{J}(\theta^k) \rvert < \varepsilon$:
		\begin{enumerate}
			\item Solve Problem \eqref{proto} and Problem \eqref{adj} to evaluate $\nabla \mathscr{J}(\theta^k)$.
			\item Approximate $\mathscr{J}(\theta^k)$ using the two-dimensional Simpson's method.
			\item Perform a line search to find $t_k$.
			\item Update $\theta^{k+1} = \mathcal{P}_{\mathscr{U}_{ad}^\theta}\!\left(\theta^k - t_k \nabla \mathscr{J}(\theta^k)\right)$.
			\item Increment $k$.
		\end{enumerate}
	\end{enumerate}
\end{algorithm}
\begin{figure}[H]
	\begin{minipage}[b]{0.4\textwidth}
		\includegraphics[width=\linewidth]{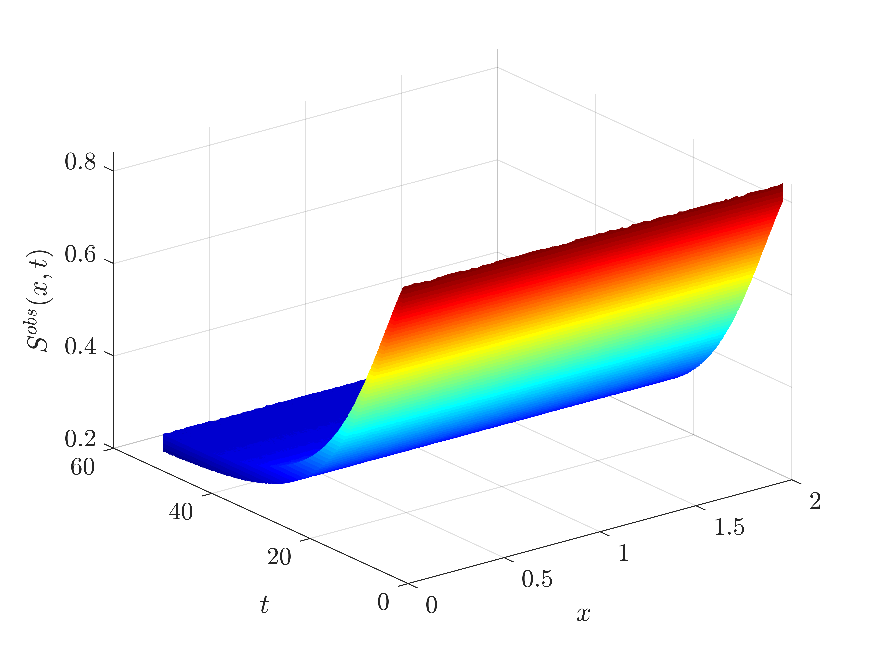} 
	\end{minipage}
	\begin{minipage}[b]{0.4\textwidth}
		\includegraphics[width=\linewidth]{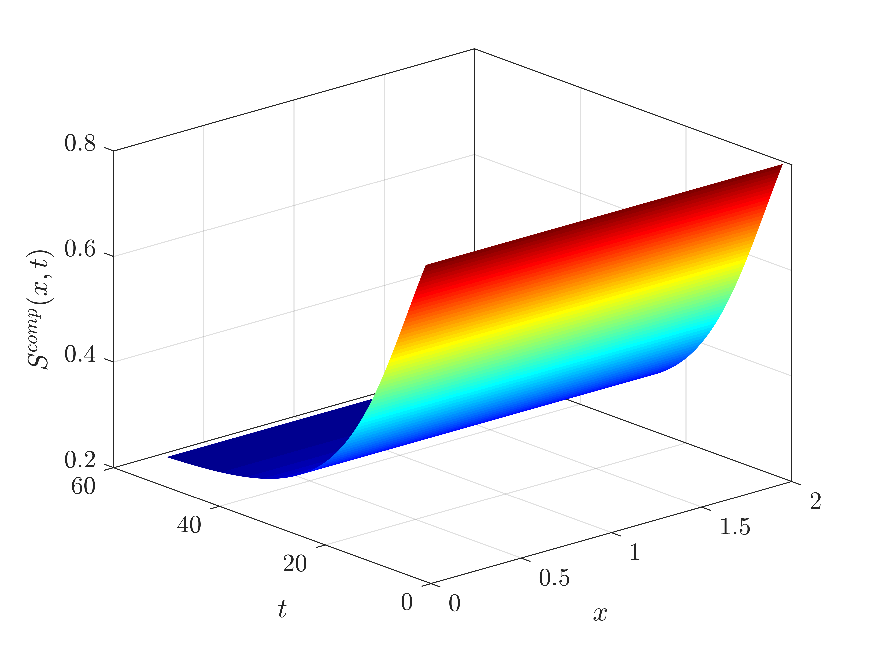} 
	\end{minipage}
	\caption{Noisy experimental observation of the susceptible population (left) and computed density of the susceptible population (right).}
	\label{fig4}
\end{figure}
\begin{figure}[H]
	
	\begin{minipage}[b]{0.4\textwidth}
		\centering
		\includegraphics[width=\linewidth]{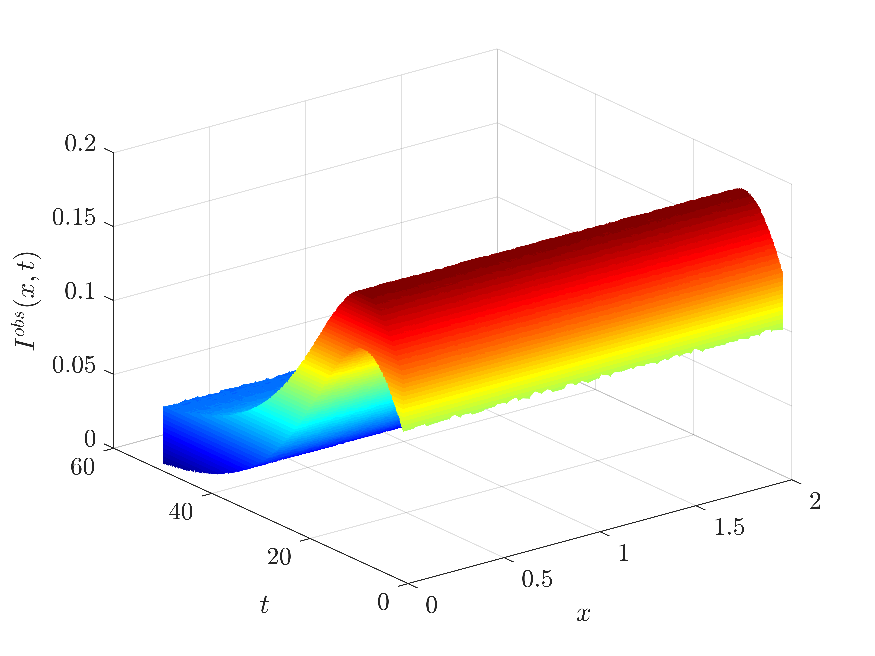} 
	\end{minipage}
	\begin{minipage}[b]{0.4\textwidth}
		\centering
		\includegraphics[width=\linewidth]{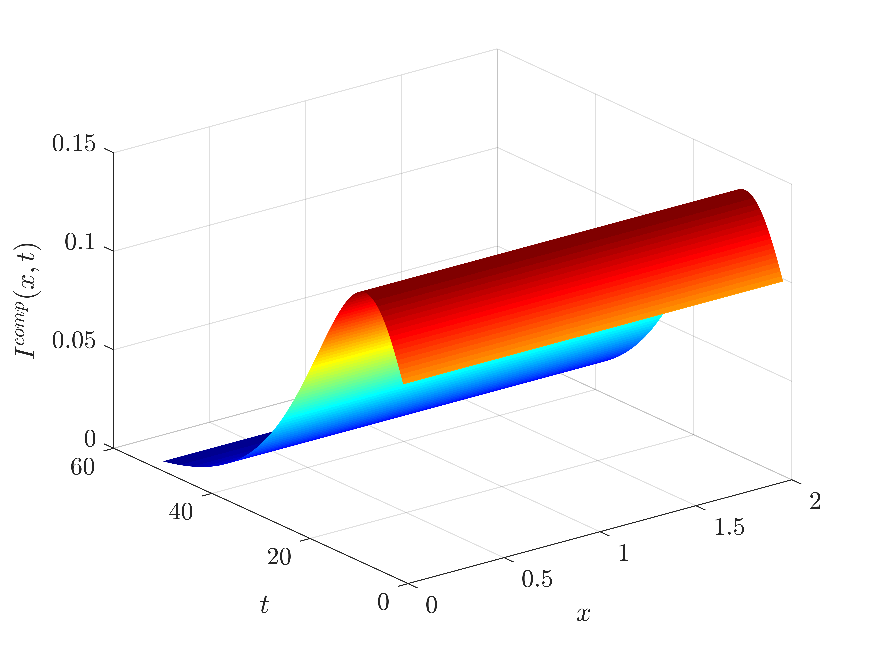} 
	\end{minipage}
	\caption{Noisy experimental observation of the infected population (left) and computed density of the infected population (right).}
	\label{fig5}
\end{figure}
\begin{figure}[H]
	\includegraphics[width=1\linewidth]{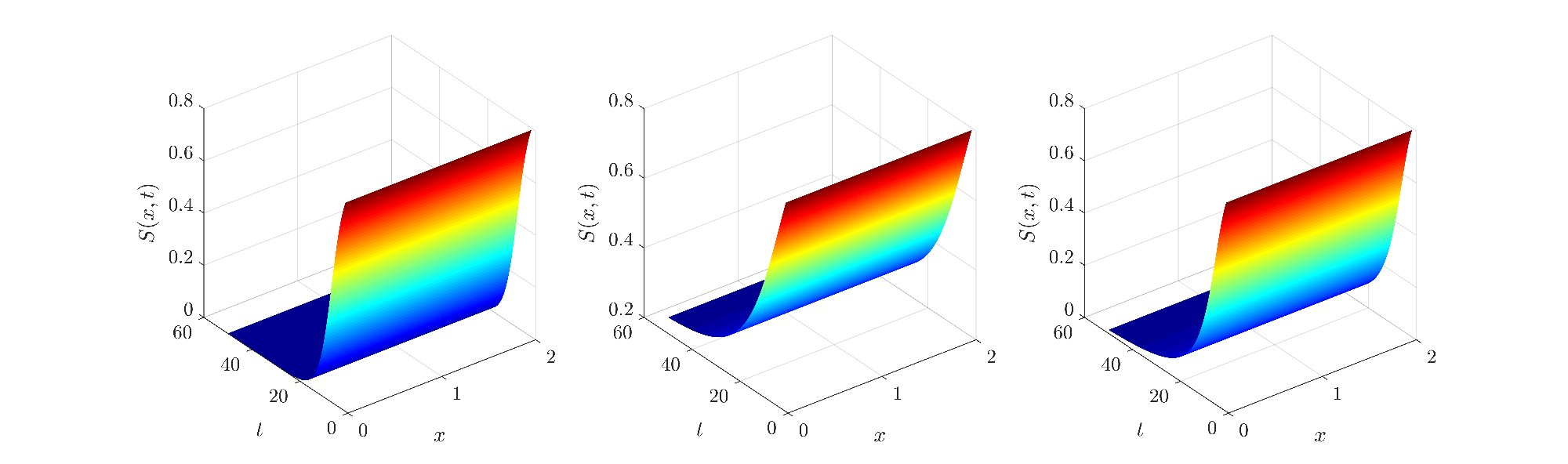} 
	\centering
	\caption{Comparison between the generated susceptible densities for independent simulations in the case $m=1$: $(S,I)\mapsto 0.4 S I$ (left), $(S,I)\mapsto \dfrac{0.4 SI}{1+I}$ (middle) and $(S,I)\mapsto \dfrac{0.4 SI}{S+I+R}$ (right).}
	\label{fig1}
\end{figure}
\begin{figure}[h]
	\centering
	\includegraphics[width=1\linewidth]{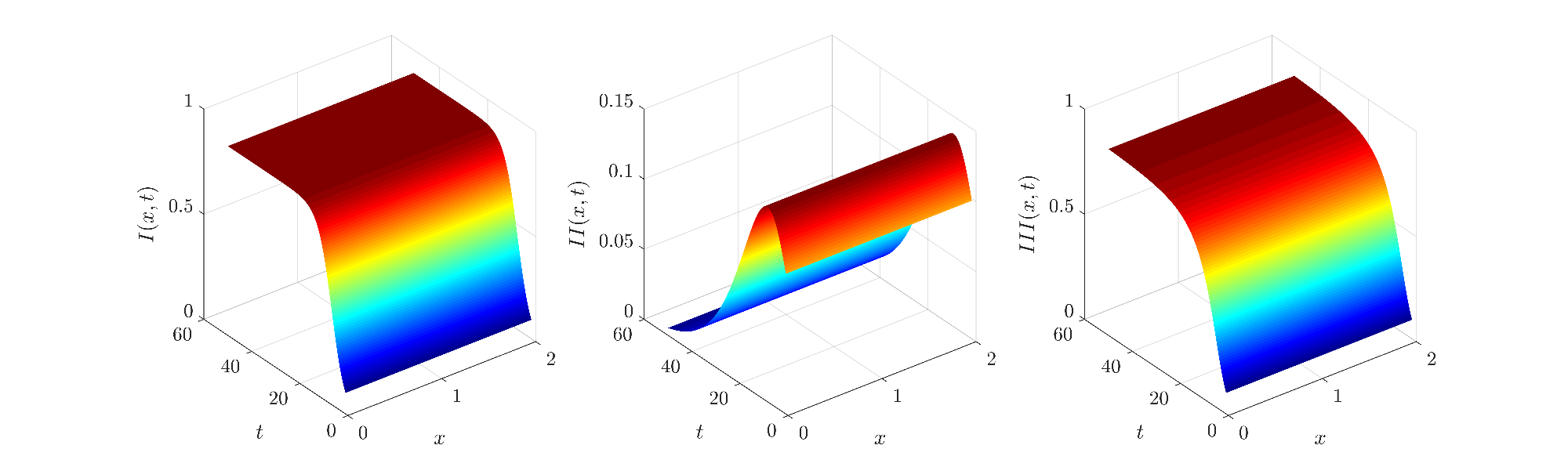} 
	\caption{Comparison between the generated infected densities for independent simulations in the case $m=1$: $(S,I)\mapsto 0.4 S I$ (left), $(S,I)\mapsto \dfrac{0.4 SI}{1+I}$ (middle) and $(S,I)\mapsto \dfrac{0.4 SI}{S+I+R}$ (right).}
	\label{fig2}
\end{figure}
\begin{figure}[h]
	\centering
	\includegraphics[width=0.4\linewidth]{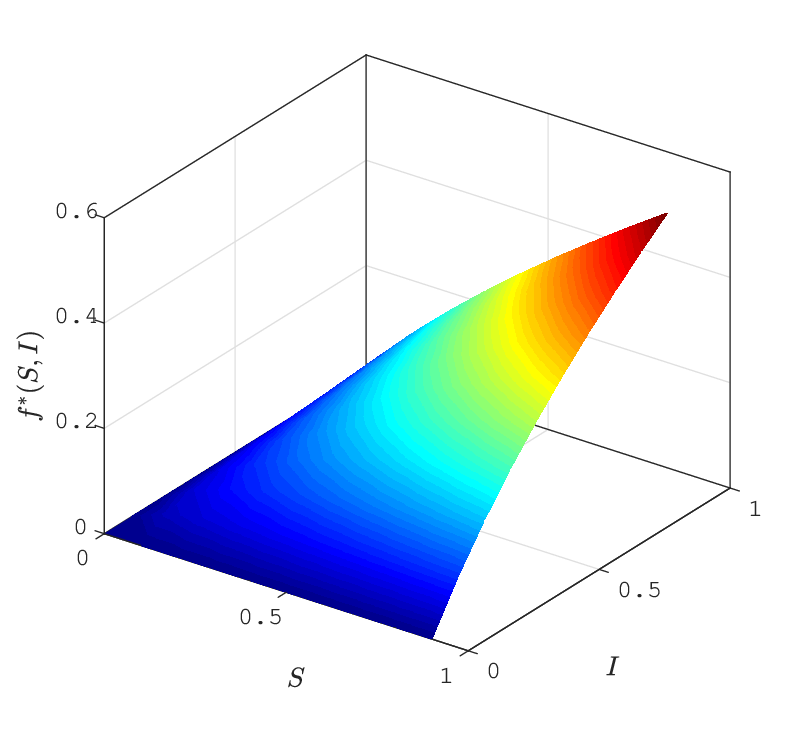} 
	\caption{Optimal incidence function corresponding to the generated experimented data.}
	\label{fig20}
\end{figure}
\begin{figure}[H]
	\centering
	\includegraphics[width=1\linewidth]{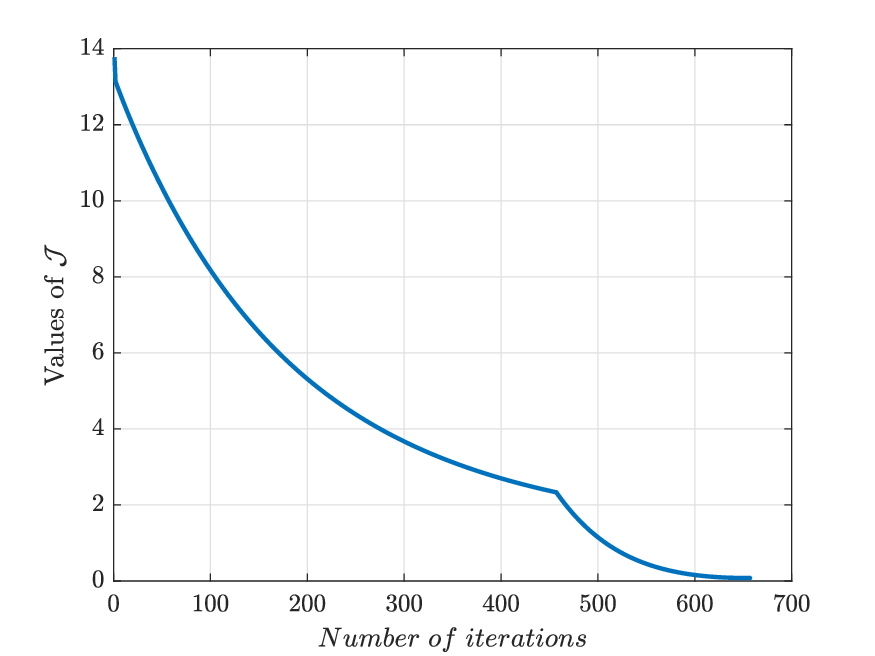}
	\caption{Evolution of the cost functional with respect to the number of iterations.}
	\label{fig3}
\end{figure}

\begin{remark}
	Note that in the second step of Algorithm \ref{myalgo}, we are assuming that real observations are taking the form of a small random perturbation of the solution corresponding to Model \eqref{proto}. Although it seems restrictive, such an assumption is well-known to be used in the literature when it comes to the validation of theoretical approaches pertaining to inverse problems. Actually, the practical validity of such an assumption relies on how well the considered model is theoretically built. For such cases, we refer the reader for example to \cite[Section 5]{yang2018landweber}. 
\end{remark}
\section{Discussion and Future Work}
\label{section5}
Extensive research in spatio-temporal epidemiology, where various incidence functions have been independently used to capture several biological characteristics, has raised the question of how one can determine which incidence function is more suitable. This paper presented a PDE-constrained optimization approach allowing to identify the weights at which a given set of incidence functions contribute to the dynamics of a certain reaction--diffusion epidemic model. We point out that the approach developed in this paper can be extended to other  contexts such as ecological models, where the functional response can take various forms \cite{mehdaoui2024new}, and economic models where the choice of the production function imposes a major problem \cite{mehdaoui2025state}. Let us also mention that the techniques developed here followed by the ones in \cite{xiang2015solving,coronel2019note,coronel2021existence,chang2024time} can provide a fresh perspective on the calibration of reaction--diffusion epidemic models, by simultaneously optimizing the choice of the incidence function and its parameters. This  two-stage optimization problem is of a great interest and serves as a future research direction. On the other hand, we mention that the idea of considering a convex combination of incidence functions in epidemic models as in Model \eqref{proto} has not been explored in basic cases (absence of diffusion). Thus, incorporating such a convex combination in such cases and analyzing the effects of weights on stability and asymptotic behavior are also promising research directions, which we will explore in future work.
\section*{Statements and Declarations}
\subsection*{Data availability}
Not applicable.
\subsection*{Funding}
No funding was received.
\subsection*{Conflict of interest}
The authors declare that they have no conflicts of interest.

\bibliographystyle{apalike}
\bibliography{MT_sp}	
\end{document}